\documentclass[10pt]{amsart}

\usepackage{amssymb,hyperref}
\usepackage{amsmath}
\usepackage{amsfonts}
\usepackage[all]{xy}

\newtheorem{thm}{Theorem}[section]

\newtheorem{cor}[thm]{Corollary}	

\newtheorem{defi}[thm]{Definition}
\newtheorem{remark}[thm]{Remark}
\newtheorem{example}[thm]{Example}
\numberwithin{equation}{section}

\newcommand{\C}{\mathbb{C}}
\newcommand{\R}{\mathbb{R}}

\newcommand{\Z}{\mathbb{Z}}
\newcommand{\g}{\mathfrak{g}}

\begin{document} 

\title[ ]{Hitchin sections  on $3$-dimensional Sasakian manifolds}

\author[H. Kasuya]{Hisashi Kasuya}

\address{Graduate School of Mathematics,
Nagoya University}

\email{kasuya@math.nagoya-u.ac.jp}



\begin{abstract}
By using non-abelian Hodge correspondence on compact Sasakian manifolds, we investigate analogous constructions of Hitchin sections of rank $2$ on $3$-dimensional  compact Sasakian manifolds.
We obtain    canonical deformations of $\widetilde{SL_{2}(\R)}$-Sasakian structures giving a diffeomorphism from  the vector space of 
basic quadratic differentials onto a connected component of the space of equivalent classes of sasakian structures of negative basic first Chern classes and topologically trivial CR structures equipped with an appropriate manifold structure.

\end{abstract}

\maketitle

\section{Introduction}
 On a compact Riemannian surface   $\Sigma$ of genus $g(\Sigma)\ge 2$,
 as a non-abelian version  of Hodge theory 
\[H^{1}(\Sigma, \Z)\otimes \C=H^{1,0}(\Sigma)\oplus H^{0,1}(\Sigma),
\]
the non-abelian Hodge correspondence (\cite{Cor, Hit, Si1, Si2}) implies  that 
moduli spaces ${\mathcal M}_{B}$ of representations of the fundamental group $\pi_{1}(\Sigma,x)$ (Betti Moduli)   are "isomorphic" to moduli spaces ${\mathcal M}_{Dol}$ of Higgs bundles (Dolbeault Moduli).
As an analogue of the  projection $H^{1}(\Sigma, \Z)\otimes \C=H^{1,0}(\Sigma)\oplus H^{0,1}(\Sigma)\to H^{1,0}(\Sigma)$, applying the non-abelian Hodge correspondence we obtain the canonical maps $\tau: {\mathcal M}_{B}\to V$ from ${\mathcal M}_{B}$ to vector spaces $V$ of certain dimensions.
Such canonical maps are called the  Hitchin fibrations.
For a "nice" moduli space ${\mathcal M}_{B}$, the Hitchin fibration $\tau: {\mathcal M}_{B}\to V$ is proper (\cite{Hit, Si3, Si4}).
In particular, in case $ {\mathcal M}_{B} $ is the moduli space of  representations into $PSL_{2}(\C)$, the Hitchin fibration $\tau: {\mathcal M}_{B}\to V$ is proper and surjective.
In this case, $V$ is the vector space of quadratic differentials on $\Sigma$ which is  a real vector space of dimension $6g(\Sigma)-6$.
Moreover, in this case, Hitchin (\cite{Hit}) constructs a section $\sigma:V=\R^{6g(\Sigma)-6}\to {\mathcal M}_{B}$ and $\sigma$ sends $\R^{6g(\Sigma)-6}$ diffeomorphically onto a connected component of  the moduli space  ${\rm Hom}^s(\pi_{1}(\Sigma,x), SL_{2}(\R))/SL_{2}(\R)$ of simple $SL_{2}(\R)$-representations.
This gives a new interpretation of the classical fact: deformations of hyperbolic  Riemannian surfaces corresponds to quadratic differentials.
In \cite{Hit2}, Hitchin generalize this section to certain  higher rank cases.
Recently constructions   of sections of Hitchin fibrations for representations of $\pi_{1}\Sigma$ into several Lie groups and geometric interpretations of them are actively studied (see \cite{Gar} and references therein).

We develop non-abelian Hodge correspondence on compact Sasakian manifolds which are  odd-dimensional counter parts of K\"ahler manifolds (\cite{BK, BK2}).
The purpose of this paper is to focus on the non-abelian Hodge correspondence on $3$-dimensional  Sasakian manifolds $M$  for  studying  sections of "Hitchin fibrations" $\tau: {\rm Hom}^s(\pi_{1}(M,x), SL_{2}(\C))/SL_{2}(\C) \to H^{1,0}_{B}(M, T^{1,0\ast}_M)$ given in \cite{Ka}.
We can obtain the following result by analogous arguments  in \cite{Hit, Hit2} that is the non-abelian Hodge correspondence for cyclic Higgs bundles (see \cite{Li}).

\begin{thm}\label{hisec}
Let $(M,\, T^{1,0}_{M} ,\, \eta) $ be a compact  $3$-dimensional Sasakian manifold such that  $c_{1, B}(T^{1,0})=-C[d\eta]$ for some positive constant $C$ and 
 $c_{1}(T^{1,0}M)=0$ i.e. $T^{1,0}M$ is  smoothly   trivial.
 There exists a smooth 
map  $\sigma:H^{1,0}_{B}(M, T^{1,0\ast}_M)\to  {\rm Hom}^s(\pi_{1}M, SU(1,1)))/SU(1,1)$ which gives a diffeomorphism   from a vector space $H^{1,0}_{B}(M, T^{1,0\ast}_M)$ onto a connected component of ${\rm Hom}^s(\pi_{1}(M,x), SU(1,1)))/SU(1,1)$.

\end{thm}
We notice that $SL_{2}(\R)\cong SU(1,1)$.
It is known  that  a compact  $3$-dimensional Sasakian  manifold $M$ under the assumption in  Theorem \ref{hisec} is a Seifert fibration over a complex $1$-dimensional orbifold $X$.
$H^{1,0}_{B}(M, T^{1,0\ast}_M)$ is the space of basic quadratic differentials on $M$ and which is isomorphic to the space of quadratic differentials on  $X$.
We can say that the real dimension of $H^{1,0}_{B}(M, T^{1,0\ast}_M)$ is $6g(X)-6+2n(X)$ where $g(X)$ is  the genus of  $X$ and $n(X)$  the number of orbifold points of $X$.
Thus Theorem \ref{hisec} can be seen as an orbifold analogue of Hitchin sections (cf. \cite{NS}).

We are more interested in Sasakian geometric interpretations of the map
\[\sigma:H^{1,0}_{B}(M, T^{1,0\ast}_M)\to  {\rm Hom}^s(\pi_{1}M, SU(1,1)))/SU(1,1).\]
The key idea is to take a "lifting".
Consider the universal covering group   $p: \widetilde{SL_{2}(\R)}\to SL_{2}(\R)\cong SU(1,1)$.
Denote by ${\mathcal R}(\pi_{1}(M,x), \widetilde{SL_{2}(\R)})$ the  subset in  ${\rm Hom}(\pi_{1}(M,x), \widetilde{SL_{2}(\R)})$ consisting of  representations  $\rho\in {\rm Hom}(\pi_{1}(M,x), G)$ such that $\rho$ are  faithful and  have cocompact discrete images.
Let  \[\pi: {\mathcal R}(\pi_{1}(M,x), \widetilde{SL_{2}(\R)})/\widetilde{SL_{2}(\R)}\to {\rm Hom}^s(\pi_{1}(M,x), SU(1,1))/SU(1,1)
\]
be the map associated with the covering map $p: SL_{2}(\R)\cong SU(1,1)$.

\begin{thm}
Under the same assumption in Theorem  \ref{hisec},
 there exists a 
 map  $\widetilde{\sigma}: H^{1,0}_{B}(M, T^{1,0\ast}_M)\to {\mathcal R}(\pi_{1}(M,x), \widetilde{SL_{2}(\R)})/\widetilde{SL_{2}(\R)}$ such that $\pi\circ \widetilde{\sigma}=\sigma$ and $\widetilde{\sigma}$ is a diffeomorphism  from a vector space $H^{1,0}_{B}(M, T^{1,0\ast}_M)$ onto a connected component of ${\mathcal R}(\pi_{1}(M,x), \widetilde{SL_{2}(\R)})/\widetilde{SL_{2}(\R)}$.
\end{thm}

We notice that the set $ {\mathcal R}(\pi_{1}M, \widetilde{SL_{2}(\R)})/\widetilde{SL_{2}(\R)}$ is identified with the set of "equivalence"  classes of Sasakian structures  $(T^{1,0}_{M} ,\, \eta) $ on $M$ such that $c_{1, B}(T^{1,0})=-C[d\eta]$ for some positive constant $C$ and 
 $c_{1}(T^{1,0}_M)=0$.
For a Sasakian structure $(T^{1,0}_{M} ,\, \eta) $ on  a compact  $3$-dimensional  manifold  $M$ such that $c_{1, B}(T^{1,0})=-C[d\eta]$ for some positive constant $C$ and 
 $T^{1,0}_M$ is smoothly   trivial,   $(T^{1,0}_{M} ,\, \eta) $ is "equivalent"  to  a canonical $ \widetilde{SL_{2}(\R)}$-Sasakian structure i.e. a left-invariant Sasakian structure on $\Gamma\backslash \widetilde{SL_{2}(\R)}$ (see \cite{Bel, KM} and Theorem \ref{uni}). 
The  lifted  Hitchin section $\widetilde{\sigma}: H^{1,0}_{B}(M, T^{1,0\ast}_M)\to {\mathcal R}(\pi_{1}(M,x), \widetilde{SL_{2}(\R)})/\widetilde{SL_{2}(\R)}$ is regarded as deformations of a $ \widetilde{SL_{2}(\R)}$-Sasakian structure.
$\widetilde{\sigma}: H^{1,0}_{B}(M, T^{1,0\ast}_M)\to {\mathcal R}(\pi_{1}(M,x), \widetilde{SL_{2}(\R)})/\widetilde{SL_{2}(\R)}$ is induced by the following.

\begin{thm}[\rm see Theorem \ref{secdef} for the detail]
Under the same assumption in Theorem  \ref{hisec}, 
there exists a canonical  smooth family $(\eta_{\alpha}, w_{\alpha})$ of inequivalent  $\widetilde{SL_{2}(\R)}$-Sasakian structures on $M$ parametrized by $ H^{1,0}_{B}(M, T^{1,0\ast}_M)$.
\end{thm}
We notice that this smooth family $(\eta_{\alpha}, w_{\alpha})$ gives an  differentiable   deformation  of the  transverse holomorphic structure associated with the original Sasakian structure $(T^{1,0}_{M} ,\, \eta) $ such  that the Kodaira-Spencer map at the origin  is an isomorphism.
This means  that   the entire  space of infinitesimal deformations of the transverse holomorphic structure gives    a canonical local  coordinate  of the space of equivalence classes of Sasakian structures.

\noindent {\bf Acknowledgement.}  
The author would like to  thank Qiongling Li and Takashi Ono for helpful comments.
The author also would like to thank Hiraku Nozawa for answering author's  questions.
\section{Sasakian geometry}

\subsection{Sasakian structures}

Let $M$ be a $(2n+1)$-dimensional real smooth manifold. A {\em CR-structure} on $M$ is an $n$-dimensional 
complex involutive sub-bundle $T^{1,0}_{M}$ of the complexified tangent bundle $TM_{\C}\,=\, 
TM\otimes_{\mathbb R} {\C}$ such that $T^{1,0}_{M}\cap T^{0,1}_{M}\,=\,\{0\}$, where 
$T^{0,1}_{M}\,=\,\overline{T^{1,0}_{M}}$.
 Let 
$S\,:=\,TM\cap (T^{1,0}_{M}\oplus T^{0,1}_{M})\subset TM$.
We have the almost complex structure $I\,:\, S\,\longrightarrow
\, S$ associated to the decomposition $ S_{\C}\,=\,T^{1,0}_{M}\oplus T^{0,1}_{M}$.

A {\em 
strongly pseudo-convex CR structure} on $M$ is a pair $(T^{1,0}_{M} ,\, \eta) $ consisting
of a CR structure $T^{1,0}_{M}$ and a contact 
$1$-form $\eta$ such that $\ker\eta\,=\,S$ and the bilinear form on $S$ defined by 
$L_{\eta}(X,\,Y)\,=\,d\eta(X, IY)$ is a Hermitian metric on $(S,\, I)$.

A {\em Sasakian structure } is a strongly pseudo-convex CR structure $(T^{1,0}_{M} ,\, \eta) $ such that 
 for  the Reeb vector field $\xi$ associated to the contact $1$-form $\eta$, 
 for any smooth section $X$ of $T^{1,0}_{M}$ the Lie bracket 
$[\xi,\, X]$ is also a section of $T^{1,0}_{M}$.
We consider the $1$-dimensional foliation ${\mathcal F}_{\xi}$ on $M$ generated by $\xi$.
If $(T^{1,0}_{M} ,\, \eta) $ is Sasakian, then $(\xi, T^{1,0}_{M})$ can be seen as  a transverse holomorphic structure
on a foliated manifold  $(M,{\mathcal F}_{\xi})$ and $d\eta$ is a transverse K\"ahler form on ${\mathcal F}_{\xi}$.

Let $(M,\, T^{1,0}_{M} ,\, \eta) $ be a Sasakian manifold.
Consider the unique homomorphism $\Phi_{\xi}\,:\, TM\,\longrightarrow\, TM $ that extends
$I\,:\, S\,\longrightarrow\, S$ satisfies the condition $\Phi_{\xi}(\xi)\,=\,0$.
Then, $\Phi^{2}_{\xi}\,=\,-{\rm Id}+\xi\otimes \eta$.
We call $\Phi_{\xi}$ the {\em almost contact structure} associated
to a Sasakian structure $(T^{1,0}_{M} ,\, \eta) $.
We define the Riemannian metric $$g_{\eta}(X,\,Y)\,=\,\eta(X)\eta(Y)+ d\eta(X,\, \Phi_{\xi} Y).$$
We call $g_{\eta}$ the {\em Sasakian metric} associated to the Sasakian structure $(T^{1,0}_{M} ,\, \eta) $.

\subsection{Basic forms}
Let $(M,\, T^{1,0}_{M} ,\, \eta) $ be a Sasakian manifold.
A differential form $\omega$ on $M$ is called {\em basic} if the equations
\[
i_{\xi}\omega\,=\,0\,=\, i_{\xi} d\omega
\]
hold. We denote by $A^{\ast}_{B}(M)$ the subspace of basic 
forms in the de Rham complex $A^{\ast}(M)$. Then
$A^{\ast}_{B}(M)$ is a sub-complex of the de Rham complex $A^{\ast}(M)$. Denote by $H_{B}^{\ast}(M)$
the cohomology of the basic de Rham complex $A^{\ast}_{B}(M)$. We note that $d\eta\,\in\,
A^{2}_{B}(M)$ and $[d\eta]\,\not=\,0\,\in\, H_{B}^{2}(M)$ if $M$ is compact.
We have the bigrading $A^{r}_{B}(M)_{\C}\,=\,\bigoplus_{p+q=r} A^{p,q}_{B}(M)$ as well as the decomposition
of the exterior differential $$d\big\vert_{A^{r}_{B}(M)_{\C}}\,=\,\partial_{B}+\overline\partial_{B}$$
on $A^{r}_{B}(M)_{\C}$, so that $$\partial_{B}\,:\,A^{p,q}_{B}(M)\,\longrightarrow\,
A^{p+1,q}_{B}(M)\ \ \text{ and }\ \
\overline\partial_{B}\,:\,A^{p,q}_{B}(M)\,\longrightarrow\,
A^{p,q+1}_{B}(M).$$ 
Denote by $H_{B}^{\ast,\ast}(M)$
the cohomology of the basic Dolbeault  complex $(A^{\ast,\ast}_{B}(M),\bar\partial_{B})$.
We  define the real operator $d_{B}^{c}=-\sqrt{-1}(\partial_{B}-\overline\partial_{B}): A^{r}_{B}(M)\to A^{r+1}_{B}(M)$.
We note that $d\eta\,\in\, A^{1,1}_{B}(M)$.

We briefly  review  transverse Hodge theory (\cite{KT}, \cite{EKA}).
Consider the usual Hodge star operator $\ast\,:\, A^{r}(M)\,\longrightarrow\, A^{2n+1-r}(M)$
associated to the Sasakian metric $g_{\eta}$ and the formal adjoint operator
$\delta\,=\,-\ast d\ast \,:\,A^{r}(M)\,\longrightarrow\, A^{r-1}(M)\, .$
We define the basic Hodge star operator  $\star_{B}\,:\, A^{r}_{B}(M)
\,\longrightarrow\,A^{2n-r}_{B}(M)$ by 
$\star_{B}\omega\,=\,\ast(\eta\wedge \omega)$ for $\omega\,\in\, A^{r}_{B}(M)$.
Also define the operators $$\delta_{B}\,=\,-\star_{B}d\star_{B}\,:\,
A^{r}_{B}(M)\,\longrightarrow\, A^{r-1}_{B}(M)\, ,$$
$$\partial_{B}^{\ast}\,=\,-\star_{B}\overline\partial_{B}\star_{B}\,:\,
A^{p,q}_{B}(M)\,\longrightarrow\, A^{p-1,q}_{B}(M)\, ,$$
$$\overline\partial_{B}^{\ast}\,=\,-\star_{B}\partial_{B}\star_{B}\,:\,
A^{p,q}_{B}(M)\,\longrightarrow\, A^{p,q-1}_{B}(M)$$ and
$\Lambda \,=\,-\star_{B}(d\eta\wedge)\star_{B}$.
They are the formal adjoints of $d$, $\partial_{B}$, $\overline\partial_{B}$ and
$(d\eta\wedge)$ respectively for the pairing 
\[A^{r}_{B}(M)\times A^{r}_{B}(M)\,\ni\,
(\alpha,\,\beta)\,\longmapsto\, \int_{M} \eta\wedge\alpha\wedge \star_{B}\beta\, .
\]
Define the Laplacian operators $$\Delta\,=\,d\delta+\delta d:\, A^{r}(M)\,\longrightarrow\, A^{r}(M)\ \ \text{ and }
\ \ \Delta_{B}=d\delta_{B}+\delta_{B}d\,:\,A^{r}_{B}(M)\,\longrightarrow\,
A^{r}_{B}(M).$$ 
For $\omega\,\in\, A^{r}_{B}(M)$, since the relation $\ast\omega
\,=\,(\star_{B}\omega)\wedge \eta$ holds, we have the relation 
\[\delta\omega\,=\,\delta_{B}\omega+\ast (d\eta\wedge \star_{B}\omega)\, .
\]
Thus, for $\omega\,\in\, A^{1}_{B}(M)$, the equality $\delta_{B}\omega
\,=\,\delta\omega$ holds, and hence for $f\,\in\, A^{0}_{B}(M)$, we have
that $\Delta_{B}f\,=\,\Delta f$.
The basic K\"ahler identities
\[[\Lambda, \partial_{B}]\,=\,
-\sqrt{-1}\overline\partial_{B}^{\ast}\ \ \text{ and } \ \
[\Lambda ,\overline\partial_{B}]=\sqrt{-1}\partial_{B}^{\ast}
\]
hold, and these imply the equations
\[\Delta_{B}\,=\,2\Delta_{B}^{\prime}\,=\,2\Delta_{B}^{\prime\prime}\, ,
\]
where $\Delta_{B}^{\prime}\,=\,\partial_{B}\partial_{B}^{\ast}+\partial_{B}^{\ast}\partial_{B}$
and $\Delta_{B}^{\prime\prime}\,=\,\overline\partial_{B}\overline\partial_{B}^{\ast}+
\overline\partial_{B}^{\ast}\overline\partial_{B}$.

\subsection{Equivalence  of Sasakian structures}
A {\em Sasakian-isomorphism} between two Sasakian manifolds is a Contact-CR-diffeomorphism which is equivalent to an isometry preserving the Reeb vector field (see \cite{KM}).
A {\em rescaling} of $( T^{1,0}_{M},\eta)$ is a Sasakian structure $( T^{1,0}_{M\tau},R\eta)$ for a real number $R>0$.
The Reeb vector field of $R\eta$ is $\frac{1}{R}\xi$.
An {\em almost-isomorphism} is between two Sasakian manifolds a diffeomorphism commuting with almost contact structures associated with Sasakian manifolds which is 
equivalent to a Sasakian isomorphism up to rescalings (see \cite{KM}).
We might prefer almost-isomorphism to Sasakian-isomorphism for classifying  Sasakian manifolds with identifying  rescallings.
We   define the set ${\mathfrak S}(M)$ of Sasakian structures on a manifold $M$ as the set of  pairs $(\eta, \Phi_{\xi})$ of contact structures $\eta$ and almost contact structures $\Phi_{\xi}$ associated with Sasakian structures $(T^{1,0}_{M} ,\, \eta) $.
Then, ${\mathfrak S}(M)$ is a set of sections of vector bundle $TM^{\ast}\oplus  {\rm End}(TM)$.

Let $(M,\, T^{1,0}_{M},\,\eta)$ be a Sasakian manifold. An {\em $A^1_{B}$-deformation} of $( T^{1,0}_{M}, 
\,\eta)$ is a Sasakian structure $(T^{1,0}_{M\tau},\,\eta^{\tau})$ such that $\eta^{\tau}\,=\,\eta+\tau$ for 
$\tau\,\in\, A^1_{B}$ and $$T^{1,0}_{M\tau}\,=\,\{X-\tau(X)\xi\,\,\big\vert\ \, X\,\in\, T^{1,0}_{M}\}.$$ 
The Reeb vector field of $\eta^{\tau}$ is $\xi$, and the transverse holomorphic structure 
$(\xi, T^{1,0}_{M\tau})$  is same as 
$(\xi, T^{1,0}_{M})$.
The almost contact structure  $\Phi_{\xi}^{\tau}$ of $(T^{1,0}_{M\tau},\,\eta^{\tau})$ is $\Phi_{\xi}-\xi\otimes \tau\circ \Phi_{\xi}$.
We notice that   the basic de Rham complex $A^{\ast}_{B}(M)$  associated with  $(T^{1,0}_{M},\,\eta)$ is same as the one associated with  a  $A^1_{B}$-deformation $(T^{1,0}_{M\tau},\,\eta^{\tau})$.
We have $[d\eta]_{B}=[d\eta^{\tau}]_{B}$ in the basic cohomology $H^{2}_{B}(M)$.
By $dd^{c}_{B}$-Lemma we have a basic function $f\in A^{0}_{B}(M)$ such that $d(\eta-\eta^{\tau})=dd^c_{B}f$.
Hence, we have a closed $1$-form $\phi\in A^1_{B}$ such that $\tau=d^c_{B}f+ \phi$.
 If another contact structure $\eta^{\prime}$ has the Reeb vector field $\xi$, and 
$d\eta^{\prime}$ is transverse K\"ahler for $T^{1,0}_{M}\oplus \langle \xi\rangle$, then we have a 
$A^1_{B}$-deformation  $( T^{1,0}_{M\tau},\,\eta^{\tau})$ for $\tau\,=\,\eta-\eta^{\prime}$.

We say that  two Sasakian structures  $(\eta, \Phi_{\xi}),  (\eta^{\prime}, \Phi^{\prime}_{\xi^{\prime}}) \in  {\mathfrak S}(M)$ are {\em equivalent} if an  $A^1_{B}$-deformation  $(\eta^{\tau}, \Phi_{\xi}^{\tau})$ of  $(\eta, \Phi_{\xi})$ is almost-isomorphic to  $(\eta^{\prime}, \Phi^{\prime}_{\xi^{\prime}})$.

\subsection{Deformation theory of  transverse holomorphic structures}

Notice that $A^1_{B}$-deformations  of  $(T^{1,0}_{M},\,\eta)$   do not change the  transverse holomorphic structure $(\xi,  T^{1,0}_{M})$.
We consider  the Kodaira-Spencer deformation theory of transverse holomorphic structures in \cite{GHS}.
As an analogue of infinitesimal deformations of complex structures, the space of  infinitesimal deformations of a transverse holomorphic structure $(\xi, \xi\oplus T^{1,0}_{M})$ is the finite-dimensional cohomology  $H^{1}(M, \Theta)$ with values in the sheaf $\Theta$ of transversely holomorphic vector fields.
For a   differential (resp. holomorphic)  deformation family   $(\xi_{t}, \xi_{t}\oplus T^{1,0}_{M_{t}})_{t\in B}$ of  transverse holomorphic structures over a smooth manifold (resp. complex analytic space) $B$, we have the canonical $\R$-linear  (resp $\C$-linear) map  
 $T_{0}B\to  H^{1}(M, \Theta)$.
This map is called the Kodaira-Spencer map for a deformation family  $(\xi_{t}, \xi_{t}\oplus T^{1,0}_{M_{t}})_{t\in B}$.

Assume that a Sasakian structure $(T^{1,0}_{M},\,\eta)$ is {\em quasi-regular} that is the Reeb vector field $\xi$ generates a locally free $S^{1}$-action on $M$.
The quotient space $X=M/S^{1}$ admits a canonical complex orbifold structure.
We consider the  Kodaira-Spencer-Kuranishi deformation theory of the  complex orbifold $X$.
If $H^{0}(X, T^{1,0}X)=0$ and $H^{2}(X, T^{1,0}X)=0$,  then $H^{1}(M, \Theta)\cong H^{1}(X, T^{1,0}X)$ 
 (see \cite[Section  4.2]{GHS}) and hence   infinitesimal deformations of the  transverse holomorphic structure $(\xi, \xi\oplus T^{1,0}_{M})$ are  identified with the liftings of infinitesimal deformations of the complex orbifold $X$.

\begin{remark}
Unlike the Kodaira-Spencer  stability theorem (\cite{KS}) of K\"ahler structures under small deformations, the existence of compatible Sasakian structures of small deformations of a  transverse holomorphic structure associated with a Sasakian structure does not hold in general.
If  $H^{0,2}_{B}(M)=0$ (in particular $\dim M=3$), in \cite{Noz}, Nozawa proves  the existence of compatible Sasakian structures of any deformations of a  transverse holomorphic structure.

\end{remark}

\section{$G$-structure}
Let $G$ a simply connected Lie group and $\g$ the Lie algebra of $G$.
Fix a  basis $e_{1},\dots, e_{n}$ of $\g$.
A $G$-{\em structure} on a compact manifold $M$ is a global frame $\omega_{1}, \dots ,\omega_{n}$ of the cotangent bundle $TM^{\ast}$  such that
\[d\omega_{k}+\sum_{k}C^{k}_{ij}\omega_{i}\wedge \omega_{j}=0
\]
for  structure constants $\{C^{k}_{ij}\}$ of $\g$ associated with a basis $e_{1},\dots, e_{n}$.
It is known that  a compact manifold $M$ is diffeomorphic to $\Gamma\backslash G$ for some lattice $\Gamma\subset G$.
More precisely, defining $\omega =\sum \omega_{i}\otimes e_{i}\in A^{1}(M)\otimes \g$, for the universal covering $\widetilde{M}$ of $M$  at a  base point $x$,  we have a canonical diffeomorphism 
 $f:\widetilde{M}\to G$ satisfying $f^{\ast}\omega_{G}=\omega $  where $\omega_{G}$ is the Maurer-Cartan form on $G$ and 
we have the monodromy map $\rho:\pi_{1}(M,x) \to G$ such that $f$ is $\rho$-equivariant by taking the maximal integral manifold  through $(x,e)$ of   the distribution on $\widetilde{M}\times  G$ defined by $\omega-\omega_{G}$ see \cite[Section 3]{Sh}.
 For a smooth family $(\omega_{1}^t, \dots ,\omega_{n}^t)$ of $G$-structures, we obtain a smooth family $(\rho_{t})$ such that $\rho_{t}\in  {\mathcal R}(\pi_{1}(M,x), G)$ for any parameter $t$.

If we have two $G$-structures $\omega_{1}, \dots ,\omega_{n}$ and $\omega^{\prime}_{1}, \dots ,\omega^{\prime}_{n}$ which correspond  to $(f, \rho)$ and $(f^{\prime}, \rho^{\prime})$ respectively  and a diffeomorphism $\varphi$ on $M$ such that $\varphi^{\ast}\omega_{i}^{\prime}= \omega_{i} $, then by Lie's second Theorem, $\varphi$ lifts to a left-transition on $G$.
Thus, 
 the set of isomorphism class of $G$-structures on $M$ can be seen as  $ {\mathcal R}(\pi_{1}(M,x), G)/G$ where $ {\mathcal R}(\pi_{1}(M,x), G)$ is a subset in  ${\rm Hom}(\pi_{1}(M,x), G)$ consisting of   $\rho\in {\rm Hom}(\pi_{1}(M,x), G)$ such that $\rho$ is faithful,  has a cocompact discrete image and $G$ acs on ${\rm Hom}(\pi_{1}(M,x), G)$ by conjugations.


\section{Non-abelian Hodge correspondence}
\subsection{Basic vector bundles}
We define basic  structures and basic holomorphic structures  on smooth vector bundles over Sasakian manifolds $(M, T_{M}^{1,0}, \eta)$ related to ${\mathcal F}_{\xi}$ in terms of Rawnsley's partial flat connections in \cite{Ra}.

A structure of {\em basic vector bundle} on a $C^\infty$ vector bundle $E$ is a 
linear
operator $\nabla_{\xi}\,:\, {\mathcal C}^{\infty}(E)
\,\longrightarrow\,
{\mathcal C}^{\infty}(E)$ such that
\[ \nabla_{\xi}(f s)\,=\,f\nabla_{\xi} s+ \xi(f) s.
\]
A differential form $\omega\,\in\, A^{\ast}(M,\,E)$ with 
values in $E$ is called {\em basic} if 
the equations
\[
i_{\xi}\omega=0=\nabla_{\xi} \omega
\]
hold where we extend $\nabla_{\xi}: A^{\ast}(M,\,E)\to A^{\ast}(M,\,E)$.
Let $A^{\ast}_{B}(M,\, E)\,\subset\,A^{\ast}(M,\,E)$ denote the 
subspace of basic forms in the space $A^{\ast}(M,\,E)$ of differential forms with values in 
$E$.
A Hermitian metric on $E$ is called {\em basic} if it is $\nabla_{\xi}$-invariant.
Unlike the usual Hermitian metric, a basic vector bundle may not admits a basic Hermitian metric in general.
For a connection operator $\nabla : A^{\ast}(M,\,E)\to A^{\ast+1}(M,\,E)$ such that the covariant derivative of $\nabla$ along $\xi$ is $\nabla_{\xi}$, the curvature $R^{\nabla}\in A^{2}(M,\, {\rm End} (E))$ is basic i.e. $R^{\nabla}\in A^{2}_{B}(M,\, {\rm End} (E))$.
For the Chern forms $c_{i}(E,\nabla)\in A^{2i}(M)$ associated with $\nabla$, we have $c_{i}(E,\nabla)\in A_{B}^{2i}(M)$ and we define the {\em basic Chern classes} $c_{i,B}(E)\in H^{2i}_{B}(M)$ by the cohomology classes of $c_{i}(E,\nabla)$.

A structure of {\em basic holomorphic bundle} on a $C^\infty$ vector bundle $E$ is a
partial flat connection $\nabla^{\prime\prime}$ along $ \langle \xi\rangle \oplus T^{0,1}$.
A basic holomorphic vector bundle $E$ has a canonical basic bundle structure corresponding to the derivative $\nabla^{\prime\prime}_{\xi}$.
$\nabla^{\prime\prime}$ defines the linear operator $\bar\partial_{E}: A^{p,q}_{B}(M,\, E)\to 
A^{p,q+1}_{B}(M,\, E)$ so that 
$\bar \partial_{E}( f\omega )=\bar\partial _{B}f \wedge \omega +f \bar\partial_{E}( \omega )$ for $f\in A^{0}_{B}(M), \omega\in A^{p,q}_{B}(M,\, E)$ and $\bar\partial_{E}\bar\partial_{E}=0$.
Denote by $H^{p,q}_{B}(M, E)$ the $q$-th cohomology of the complex  $(A^{p,\ast}_{B}(M,\, E), \bar\partial_{E})$.
If a basic holomorphic vector bundle $E$ admits a basic Hermitian metric $h$, as complex case, we have a unique unitary connection $\nabla^{h}$ such that for any $X\in \langle \xi\rangle \oplus T^{0,1}$, $\nabla^{h}_{X}=\nabla^{\prime\prime}_{X}$.

\begin{example}\label{line}
Consider the ${\mathcal C}^{\infty}$-trivial complex line bundle $E\,=\,M\times \C\, 
\longrightarrow\, M$ trivialized by a global frame $e$. For any $C\in\R$, we define the connection 
$
\nabla^{C}
$
on $E$ so that the connection form associated with the frame $e$ is $-2\pi\sqrt{-1}C\eta$.
Since $d\eta\,\in\, 
A^{1,1}_{B}(M)$, this $\nabla^{C}$ induces a structure of a basic holomorphic bundle.
We obtain a non-trivial basic holomorphic vector bundle structure $E_{C}$ on $E$ 
such that $c_{1,B}(E_{C})=C[d\eta]$.  
The  Hermitian metric $h$ on $E$ defined by $h(e,e)=1$  is basic on $E_{C}$. 
The 
connection $\nabla^{C}$ is unitary for $h$, and hence $\nabla^{C}=\nabla^h$. 

Define another basic Hermitian metric $\tilde{h}=Hh$ for positive basic function $H\in A^{0}_{B}(M)$.
Consider the unitary connection   $\nabla^{\tilde{h}}$ on $E_{C}$ associated with  the basic Hermitian metric $\tilde{h}$.
As usual  Holomorphic  bundles on complex manifolds, we have  $\nabla^{\tilde{h}}- \nabla^{C}=\partial_{B}{\rm log}(H)$.
Hence the  connection form of $\nabla^{\tilde{h}}$ associated with the unitary  frame $\frac{e}{H^{\frac{1}{2}}}$ is $-2\pi\sqrt{-1}C\eta- \frac{1}{2}\sqrt{-1}d_{B}^{c}{\rm log}(H)$.
\end{example}

\begin{example}

For a strongly pseudo-convex CR-manifold $(M, \,T^{1,0}_{M} ,\, \eta)$
there exists a unique affine connection $\nabla^{TW}$ on $TM$ such that the following
conditions hold (\cite{Tan, Web}):
\begin{enumerate}
\item $S$ is parallel with respect to $\nabla^{TW}$,

\item $\nabla^{TW}I\,=\,\nabla^{TW}d\eta\,=\,\nabla^{TW}\eta\,=\,\nabla^{TW}\xi\,=\,0$, and

\item the torsion $T^{TW}$ of the affine connection $\nabla^{TW}$ satisfies the equation
\[T^{TW} (X,\,Y)\,=\, -d\eta (X,\,Y)\xi
\]
for all $X,\,Y\,\in\, S_{x}$ and $x\,\in\, M$.
\end{enumerate}
This affine connection $\nabla^{TW}$ is called the {\em Tanaka--Webster connection}. 
It is known that $(T^{1,0},\,\eta)$ is a Sasakian manifold if and only if 
$T^{TW}(\xi,\,v)\,=\,0$ for all $v\, \in\, TM$. 

For a Sasakian manifold $(M, \,T^{1,0}_{M} ,\, \eta)$,
the Tanaka-Webster connection $\nabla^{TW} $ defines a 
structure of a basic holomorphic bundle on $T^{1,0}$ such that $\nabla^{TW} $ on $T^{1,0}$ is the unitary connection associated with the transverse K\"aher structure $d\eta$.

\end{example}

\subsection{Basic bundles over quasi-regular Sasakian manifolds}
Assume that a Sasakian structure $(T^{1,0}_{M},\,\eta)$ is quasi-regular that is the Reeb vector field $\xi$ generates a locally free $S^{1}$-action on $M$.
We take a $S^{1}$-action effective.
A $S^1$-equivariant  vector bundle $E$ over $M$ is  a basic  vector bundle and  corresponding to a vector orbibundle over the orbifold $X=M/S^{1}$ (see \cite{BK2}).
Assume $E$ is holomorphic  as a basic    vector bundle.
We have the basic Dolbeault complex $(A^{0,\ast}_{B}(M,E),\bar\partial_{E})$.
We notice that the sheaf on $X=M/S^{1}$ defined by  basic smooth functions on $S^1$-invariant open sets is fine 
 (see \cite{Pfl}).
Hence, by the same way as the  usual Dolbeault theorem, we have an isomorphism
\[H^{0,\ast}_{B}(M, E) \cong H^{\ast}(X,E).
\]

\subsection{Higgs bundles over Sasakian manifolds}
Let $(M,T^{1,0}_{M}, \eta)$ be a compact Sasakian manifold.
A \textit{basic Higgs bundle} over $M$ is a pair
$(E, \,\theta)$ consisting of a basic holomorphic vector bundle $E$
and 
$\theta\,\in\, A^{1,0}_{B}(M,\,{\rm 
End}(E))$ satisfying the following two conditions:
$$\overline\partial_{{\rm 
End}(E)}\theta \,=\,0\ \ \text{ and }
\ \ \theta\wedge \theta\,=\,0\, .$$

We define the {\it degree} of a basic holomorphic vector bundle $E$ by 
$${\rm deg}(E)\,:=\,\int_{M}c_{1, B}(E)\wedge (d\eta)^{n-1}\wedge\eta\, .$$
Denote by ${\mathcal O}_{B}$ the sheaf of holomorphic functions on $M$, and for a holomorphic vector bundle $E$ on $M$, denote by 
${\mathcal O}_{B}(E)$ the sheaf of holomorphic sections of $E$. 
Consider ${\mathcal 
O}_{B}(E)$ as a coherent ${\mathcal O}_{B}$-sheaf.

For a basic Higgs 
bundle $(E,\, \theta)$, a {\em sub-Higgs sheaf} of $(E,\, \theta)$ is a coherent 
${\mathcal O}_{B}$-subsheaf $\mathcal V$ of ${\mathcal O}_{B}(E)$ such that $\theta ({\mathcal V})\, \subset\, {\mathcal 
V}\otimes \Omega_{B}$, where $\Omega_{B}$ is the sheaf 
of basic holomorphic $1$-forms on $M$. By \cite[Proposition 3.21]{BH}, if ${\rm rk} 
(\mathcal V)\,<\,{\rm rk}(E)$ and ${\mathcal O}_{B}(E)/\mathcal V$ is 
torsion-free, then there is a transversely analytic sub-variety $S\,\subset\, M$ of complex 
co-dimension at least 2 such that ${\mathcal V}\big\vert_{M\setminus S}$ is given by a basic holomorphic sub-bundle $V\,\subset\,
E\big\vert_{M\setminus S}$. The degree ${\rm deg}(\mathcal V)$ can be defined by integrating $c_{1, B}(V)\wedge (d\eta)^{n-1}\wedge\eta$
on this complement $M\setminus S$.

\begin{defi}
We say that a basic Higgs bundle $(E,\, \theta)$ is {\em stable} if $E$ admits a basic Hermitian metric and for every 
sub-Higgs sheaf ${\mathcal V}$ of $(E,\, \theta)$ such that ${\rm rk} (\mathcal V)\,<\,{\rm 
rk}(E)$ and ${\mathcal O}_{B}(E)/\mathcal V$ is torsion-free, 
the inequality
\[\frac{{\rm deg}(\mathcal V)}{{\rm rk} (\mathcal V)}\,<\,\frac{{\rm deg}(E)}{{\rm rk}(E)}
\]
holds.

\end{defi}

Let $(E, \,\theta)$ be a basic Higgs bundle over a compact Sasakian manifold $M$. 
Assume that $E$ admits a basic Hermitian metric $h$.
Define $\overline\theta_{h}\,\in\, A^{0,1}_{B}(M,\,{\rm End}(E))$ by
$
h(\theta (e_{1}),\, e_{2})\,=\,h(e_{1}, \,\overline\theta_{h} (e_{2}))
$
for all $e_{1},\,e_{2}\,\in\, E_x$ and all $x\, \in\, M$.
Define the canonical connection
$
D^{h}\,=\,\nabla^{h}+\theta+\overline\theta_{h}
$
on $E$.

\begin{thm}[\cite{BK, BK2}]\label{harmonicmetric}
If a basic Higgs bundle $(E, \,\theta)$ is stable and satisfies
$$c_{1,B}(E)\,=\,0\qquad {\rm and} \qquad \int_{M} c_{2,B}(E)\wedge(d\eta)^{n-2}\wedge \eta\,=\,0,$$
then there exists a basic Hermitian metric $h$ so that the canonical connection $D^{h}$ is flat 
and such Hermitian metric is unique up to a positive constant.
Moreover, such metric $h$ is a Harmonic metric on the flat bundle $(E,D^{h})$ with respect to the Sasakian metric $g_{\eta}$ and hence the flat bundle $(E,D^{h})$ is semi-simple (\cite{Cor}).

\end{thm}
This hermitian metric is said to be a {\em harmonic metric} on $(E,\theta)$. 

\section{$\widetilde{SL_{2}(\R)}$-Sasakian structure}

Let $\widetilde{SL_{2}(\R)}$ be the simply connected Lie group whose Lie algebra is $\frak{sl}_{2}(\R)$.
$\widetilde{SL_{2}(\R)}$ is the universal covering group of $SL_{2}(\R)$.
Fix a basis \[
e_{1}\,=\, \left(
\begin{array}{cc}
1 & 0 \\
0 & -1
\end{array}
\right),\qquad e_{2}\,=\, \left(
\begin{array}{cc}
0 & 1 \\
1 & 0
\end{array}
\right)\qquad {\rm and} \qquad e_{3}\,=\, \left(
\begin{array}{cc}
0 & -1 \\
1 & 0
\end{array}
\right)
\] of  $\frak{sl}_{2}(\R)$.
We have $[e_{1},\,e_{2}]\,=\,-2e_{3}$, $[e_{1},\,e_{3}]\,=\,-2e_{2}$ and  $[e_{2},\, e_{3}] \,=\,2e_{1}$.

Let $M$ be a compact $3$-dimensional manifold.
A $\widetilde{SL_{2}(\R)}$-Sasakian structure on $M$ is a pair $(\eta, w)$ such that:
\begin{enumerate}
\item $\eta\in A^{1}(M)$, $w\in A^{1}(M)_{\C}$ and $\eta, w,\overline{w}$ is a global frame of $TM_{\C}^{\ast}$:
\item  The equations 
\[\sqrt{-1}d\eta +w\wedge \bar{w}= 0 \qquad {\rm and}\qquad 
dw +2\sqrt{-1}\eta\wedge w=0
\]
hold.

\end{enumerate}

We can easily check that  $(\eta, w)$ gives a $\widetilde{SL_{2}(\R)}$-structure $( {\rm Re}w, {\rm Im}w,\xi)$ and hence 
$M\cong \Gamma\backslash \widetilde{SL_{2}(\R)}$.
We  can also check that for  the  dual frame  $(\xi, W, \overline{W})$ of  $\eta, w,\overline{w}$, $(\eta, T^{1,0}=\langle W \rangle )$ is  a left-invariant  Sasakian structure on $M\cong \Gamma\backslash \widetilde{SL_{2}(\R)}$ and $\xi$ is the Reeb vector field of the contact form $\eta$.
 
We have $E_{-\frac{1}{\pi}}\cong T^{1,0}$.
Define the  Higgs bundle $(E, \theta)$ such that $E=E_{-\frac{1}{2\pi}} \oplus E_{\frac{1}{2\pi}}$ and 
$\theta=\left(
\begin{array}{cc}
0 & 1 \\
0  & 0
\end{array}
\right)w$. Then by the equations 
\[\sqrt{-1}d\eta +w\wedge \bar{w}= 0 \qquad {\rm and}\qquad 
dw +2\sqrt{-1}\eta\wedge w=0,
\]
we can say that the standard Hermitian metric on the ${\mathcal C}^{\infty}$-trivial vector bundle $E=E_{-\frac{1}{2\pi}} \oplus E_{\frac{1}{2\pi}}$ is a harmonic metric.
The corresponding flat bundle is $d+\omega_{G}$ and hence the monodromy representation is given by $\Gamma\subset \widetilde{SL_{2}(\R)}\to SU_{1,1}(\R)$
We can say that  the  Higgs bundle $(E, \theta)$ is stable.

We notice that any non-trivial rescalling $(T^{1,0}_{M},R\eta)$ for $R\not=1$ does not come from a $\widetilde{SL_{2}(\R)}$-Sasakian structure.
We  assume that an $A_{B}^{1}$ deformation of  $(T^{1,0}_{M},\eta)$ is also comes from a  $\widetilde{SL_{2}(\R)}$-Sasakian structure $(\eta^{\prime}, w^{\prime})$.
 For  the  dual frames  $(\xi, W, \overline{W})$  and $(\xi, W^\prime, \overline{W^\prime})$ of  $\eta, w,\overline{w}$ and $\eta^{\prime}, w^{\prime},\overline{w^{\prime}}$ respectively, we have $W^\prime=W-\eta^\prime(W)\xi$.
 By $ w^{\prime}(\xi)=\overline{w^{\prime}}(W^\prime)=0$, we have $w=w^\prime$.
 This implies $dw=-2\sqrt{-1}\eta\wedge w=-2\sqrt{-1}\eta^{\prime}\wedge w$, we have $\eta=\eta^{\prime}$ by $\eta(W)=\eta^{\prime}(W)=0$.

We assume that two   $\widetilde{SL_{2}(\R)}$-Sasakian structures $(\eta, T^{1,0})$ and  $(\eta^{\prime}, w^{\prime})$ on $M$ are Sasakian-isomorphic to each other. Then there  is a  diffeomorphism $f$ on $M$ such that $f^{\ast}\eta^{\prime}=\eta$, $df(\xi)=\xi^{\prime}$ and $df (W)=k W^{\prime}$ for some non-zero complex  function $k\in C^{\infty}(M)$.
By $d\eta=f^{\ast}d\eta^{\prime}=\sqrt{-1}f^{\ast}(w\wedge \bar{w})$, we have $k$ is unitary.
By $d f^{\ast}w^\prime  =2\sqrt{-1}\eta\wedge f^{\ast}w^{\prime}$, $\xi(k)=\overline{W}(k)=0$ and hence $k$ is basic holomorphic.
Thus $k=1$ and so $f^{\ast}\eta^{\prime}=\eta$ and $f^{\ast} w=w$.

By these arguments, we can say that the set of equivalent  classes of $\widetilde{SL_{2}(\R)}$-Sasakian structures is identified with $ {\mathcal R}(\pi_{1}M, \widetilde{SL_{2}(\R)})/\widetilde{SL_{2}(\R)}$.

\begin{thm}[cf. \cite{KM}]\label{uni}
Let $(M,T^{1,0}_{M},\eta)$ be a compact $3$-dimensional Sasakian manifold with $c_{1, B}(T^{1,0})=-C[d\eta]$.
Suppose $T^{1,0}M$ is smoothly   trivial.
Then, $M$ admits a canonical $\widetilde{SL_{2}(\R)}$-Sasakian structure $(\eta^{\prime}, w)$ which is an $A_{B}^{1}$-deformation of a rescaling of  $(T^{1,0}_{M},\eta)$.
\end{thm}

\begin{proof}
By $c_{1, B}(T^{1,0}_{M})=-C[d\eta]$, 
$E_{C}\otimes T^{1,0}_{M}$ is a holomorphic basic line bundle with $c_{1, B}(E_{C}\otimes T^{1,0}_{M})=0$.
Thus, by the standard argument, $E_{C}\otimes T^{1,0}_{M}$ is unitary  flat.
Since $E_{C}$ and  $T^{1,0}_{M}$ are topologically trivial, $L_{C}\otimes T^{1,0}_{M}$ can be seen as  a line bundle $L_{\sqrt{-1}\phi}=M\times \C$ equipped  with a flat connection $d+\sqrt{-1}\phi$ for a closed $\phi\in A^{1}_{B}(M)$.
Define the Higgs bundle $(E,\theta)$ such that  
\begin{itemize}
\item the basic  holomorphic vector bundle $E\,=\,E_{-\frac{C}{2}}\otimes L_{\frac{\sqrt{-1}\phi}{2}}\oplus
E_{\frac{C}{2}}\otimes L_{\frac{ -\sqrt{-1}\phi}{2}}$  
and 
\item $\theta\,=\,\left(
\begin{array}{cc}
0&1 \\
0 & 0
\end{array}
\right)$ where $1$ is the identity  in ${\rm Hom}(L_{\frac{C}{2}}\otimes L_{\frac{ \sqrt{-1}\phi}{2}}, L_{-\frac{C}{2}}\otimes L_{-\frac{\sqrt{-1}\phi}{2}})\otimes  T^{1,0\ast}_{M}=T^{1,0}_{M}\otimes T^{1,0\ast}_{M}$ .

\end{itemize}
By the non-abelian Hodge correspondence (\cite{BK, BK2}), we have a Hermitian metric $H$ on $E$ satisfying the  Hitchin equation.
By the similar argument on  usual cyclic Higgs bundles (see \cite{Ba, Li}), we can say that $H=\left( \begin{array}{cc}
h^{-1} & 0\\
0  & h
\end{array}
\right)$.
Take a global orthonormal flame $w$ associated with $h$ corresponding to natural global flames of $L_{\frac{C}{2} }$ and  $L_{-\frac{C}{2}}$.
Then we have the flatness of the connection 
\[d+\left(
\begin{array}{cc}
\sqrt{-1} (\pi C \eta-\frac{1}{2}d^c_{B}\log (h)+ \frac{\phi}{2}) & w \\
\overline{w} & -\sqrt{-1} (\pi C \eta-\frac{1}{2}d_{B}^c\log (h)+ \frac{\phi}{2})
\end{array}
\right).
\]
Denote $\eta^{\prime}=\pi C \eta-\frac{1}{2}d^c_{B}\log (h)+ \frac{\phi}{2}$.
We obtain a $\widetilde{SL_{2}(\R)}$-Sasakian structure $(\eta^{\prime}, w)$.

\end{proof}

\begin{cor}
Let $(M,T^{1,0}_{M},\eta)$ be a compact $3$-dimensional Sasakian manifold with $c_{1, B}(T^{1,0})=-C[d\eta]$.
Then there exists finite covering $M^{\prime} \to M$ such that  $M^{\prime}$ admits a canonical $\widetilde{SL_{2}(\R)}$-Sasakian structure $(\eta^{\prime}, w^{\prime})$ which is an $A_{B}^{1}$-deformation of a rescaling of  the lifting of the Sasakian structure $(T^{1,0}_{M},\eta)$.
\end{cor}
\begin{proof}
In the proof of Theorem \ref{uni}, we should take  $p: M^{\prime} \to M$ so that  the flat line bundle  $p^{\ast}(L_{C}\otimes T^{1,0})$ does not admits a torsion part.
\end{proof}

\begin{cor}\label{SLsas}
Let $M$ be a compact $3$-dimensional manifold.
Denote ${\mathcal S}_{(-,0)}(M)$ by  the set  of Sasakian structures $(\eta, \Phi_{\xi})$ on $M$ such that $c_{1, B}(T^{1,0})=-C[d\eta]$ for some positive $C\in \R$  and $c_{1}(T^{1,0}_{M})=0$ i.e. $T^{1,0}_{M}$  is topologically  trivial.
Define  the  equivalent relation $\sim$ on  ${\mathcal S}_{(-,0)}(M)$ by  almost-isomorphism (equivalently Sasakian-isomorphisms and rescallings) and  $A^{1}_{B}$-deformations.
Then 
 the quotient set  ${\mathcal S}_{(-,0)}(M)/\sim$  is identified with the set   $ {\mathcal R}(\pi_{1}M, \widetilde{SL_{2}(\R)})/\widetilde{SL_{2}(\R)}$ by taking al $\widetilde{SL_{2}(\R)}$-Sasakian structures as representatives of equivalent classes in  ${\mathcal S}_{(-,0)}(M)/\sim$.
 
\end{cor}

\section{Moduli spaces of flat bundles of rank $2$ on $3$-dimensional Sasakian manifolds}\label{modul}
Let $(M,T^{1,0}_{M},\eta)$ be a compact Sasakian manifold.

Let ${\mathcal M}^{s}_{flat }(SL_r)$ be the moduli space of simple flat complex vector bundles over $M$ of rank $r$ with  trivial  determinant.
Fix a simple flat bundle $(E,D)$ of rank $r$ with  trivial  determinant.
Consider the basic structure $D_{\xi}$ associated with $D$ on the smooth vector bundle $E$.
We define the moduli space  
${\mathcal M}^{s}_{Bflat }(E, D_{\xi})$ of simple flat bundles with trivial determinant and with fixed basic structure $D_{\xi}$.
Obviously, ${\mathcal M}^{s}_{Bflat }(E, D_{\xi})$ contains the isomorphism class of $(E,D)$.
In \cite{Ka}, we show that  ${\mathcal M}^{s}_{Bflat }(E, D_{\xi})$ is an open and  closed set in ${\mathcal M}^{s}_{flat }(SL_r)$.

$(E,D)$ admits harmonic metric $h$ which is  unique up to positive scaler multiplication by Corlette's Theorem \cite{Cor}.
It is known that $h$ is basic on $(E, D_{\xi})$ (see \cite{BK}).
Denote by ${\rm End}_{0}(E)$ the trace free part of $ {\rm End}(E)$ associated with $h$.
We notice that ${\mathcal M}^{s}_{flat }(SL_r)$ is a complex  analytic space whose chart constructed by the Kuranishi spaces of the DGLA $A^{\ast}(M, {\rm End}_{0}(E))$.
${\mathcal M}^{s}_{Bflat }(E, D_{\xi})$ is an analytic space whose chart constructed by the Kuranishi spaces of the DGLA $A^{\ast}_{B}(M, {\rm End}_{0}(E))$.
In \cite{Ka}, we show that the inclusion $A^{\ast}_{B}(M, {\rm End}_{0}(E))\subset A^{\ast}(M, {\rm End}_{0}(E))$ induces an open-closed embedding of ${\mathcal M}^{s}_{Bflat }(E, D_{\xi})$ into  ${\mathcal M}^{s}_{flat }(SL_r)$.

Denote by ${\rm Hom}^{s}(\pi_{1}(M,x), SL_{n}(\C))$  the set of simple representations of $\pi_{1}(M,x)$ into $SL_{n}(\C)$.
 We notice that the quotient  ${\rm Hom}^{s}(\pi_{1}(M,x), SL_{n}(\C))/ SL_{n}(\C)$ is a quasi-projective variety which is  analytically isomorphic to ${\mathcal M}^{s}_{flat }(SL_r)$.

Consider  the decomposition ${\rm End}_{0}(E)={\frak u}_{0}(E)\oplus P_{0}(E)$ such that  ${\frak u}_{0}(E)$ (resp. $P_{0}(E)$) consists of anti-self-adjoint (resp. self-adjoint) operators associated with the harmonic metric $h$.
 A simple harmonic bundle structure on $(E,h)$ is a pair $(\nabla, \phi)$ of unitary connection $\nabla$ with respect to $h$ and self-adjoint basic  $1$-form $\phi\in A^{1}_{B}(P_{0}(E))$ such that $\nabla_{\xi}=D_{\xi}$, $\nabla+\phi$ is a simple flat connection and for the operator  $\nabla^{\prime\prime}\,:\, {\mathcal C}^{\infty}(E)
\,\longrightarrow\,
{\mathcal C}^{\infty}(E\otimes ( \langle \xi\rangle \oplus T^{0,1}_{M})^{\ast})$ derived from and the $(1,0)$-part $\theta$ of $\phi$,  $\nabla$, $(\nabla^{\prime\prime}, \theta)$ is a basic Higgs bundle structure on $E$.
  Define the moduli space $\mathcal M^{s}_{Bharm}(E, D_{\xi})$ as the quotient of the space of  simple harmonic bundle structures by the  gauge group $SU_{B}(E)=SU(E)\cap GL_{B}(E)$.

Define the moduli space ${\mathcal M}^{st}_{BHiggs }(E, D_{\xi})$ of  stable basic  Higgs bundle structures on $E$ with fixed basic structure $D_{\xi}$.
By non-abelian Hodge correspondence in \cite{BK}, 
we have homeomorphisms
\[{\mathcal M}^{s}_{Bflat }(E, D_{\xi})\cong {\mathcal M}^{s}_{Bharm}(E, D_{\xi})\cong {\mathcal M}^{st}_{BHiggs }(E, D_{\xi})
\]
via the correspondences 
\[\nabla+\phi\leftarrow (\nabla, \phi)\to (\nabla^{\prime\prime}, \phi^{1,0}).
\]

We assume $\dim M=3$.
In this case, Ono (\cite{Ono}) constructs an explicit  smooth manifold structure on   ${\mathcal M}^{s}_{Bharm}(E, D_{\xi})$ derived from a real vector space $A^{1}_{B}(({\frak u}_{0}(E))\oplus A^{1}_{B}(P_{0}(E))$ such that the dimension of  ${\mathcal M}^{s}_{Bharm}(E, D_{\xi})$  is the real  dimension of the  first cohomology of the Dolebeault-Higgs complex $(A^{\ast}_{B}({\rm End}_{0}(E)), \partial_{{\rm End}_{0}(E)}+\theta)$ corresponding to  $D$ and the harmonic metric $h$.
Define the smooth map (“Hitchin fibration") \[\tau : {\mathcal M}^{s}_{Bharm  }(E, D_{\xi})\to \bigoplus ^{{\rm rank} E} _{i=2}H^{0,0}_{B}(M,  (T^{1,0\ast})^i)\] defined by   the coefficients of the characteristic polynomial of Higgs fields  associated with harmonic bundle structures $(\nabla, \phi)$.
\begin{remark}
In \cite{Ono}, Ono also  defines a canonical  hyperK\"ahler structure \[(I,J,K, \omega_{I}, \omega_{K}, \omega_{J})\]  on ${\mathcal M}^{s}_{Bharm}(E, D_{\xi})$.
One complex structure $I$ is induced  by the complex structure $A^{1}_{B}(({\frak u}_{0}(E))\oplus A^{1}_{B}(P_{0}(E))=A^{1,0}_{B}(({\frak u}_{0}(E))\oplus A^{0,1}_{B}(({\frak u}_{0}(E))\oplus A^{1,0}_{B}(P_{0}(E))\oplus A^{0,1}_{B}(P_{0}(E))$.
The map $\tau : {\mathcal M}^{s}_{Bharm  }(E, D_{\xi})\to \bigoplus ^{{\rm rank} E} _{i=2}H^{0,0}_{B}(M,  (T^{1,0\ast}_{M})^i)$ is holomorphic for this complex structure.

\end{remark}

Suppose $M$ admits a $\widetilde{SL_{2}(\R)}$-Sasakian structure $(\eta, w)$.
Then we can assume 
$M=\Gamma \backslash\widetilde{SL_{2}(\R)} $ with a left-invariant Sasakian structure $(T^{1,0}_{M}=\langle W\rangle,\eta)$.
Take $(E=M\times \C^{2}, D=d+\omega_{\widetilde{SL_{2}(\R)}})$. 
Then  the standard metric on $E=M\times \C^{2}$ is a harmonic metric and the corresponding Higgs bundle $(E,\theta)$ is given by $E=E_{-\frac{1}{2\pi}} \oplus E_{\frac{1}{2\pi}}$ and 
$\theta=\left(
\begin{array}{cc}
0 & 1 \\
0  & 0
\end{array}
\right)w$. 
Then $\tau([\nabla, \phi])=\det \phi^{1,0}$.
The map $\tau: {\mathcal M}^{s}_{Bharm  }(E, D_{\xi})\to H^{0,0}_{B}(M, (T_{M}^{1,0\ast})^2)$ is proper (see \cite[Corollary 4.4, Example 4.5]{Ka}).

We can compute  \[\dim {\mathcal M}^{s}_{Bharm  }(E, D_{\xi})=4\dim_{\C} H^{0,0}_{B}(M, (T_{M}^{1,0\ast})^2)\] by an  isomorphism of Eichler-Shimura type  (see e.g. \cite{Zu}).
Consider  the Dolebeault-Higgs complex \[(A^{\ast}({\rm End}_{0}(E)), \partial_{{\rm End}_{0}(E)}+\theta)\] as the double complex \[(A^{\ast,\ast}_{B}({\rm End}_{0}(E)), \partial_{{\rm End}_{0}(E)}, \theta).\]
Then the dimension of the  first cohomology of $(A^{\ast}({\rm End}_{0}(E)), \partial_{{\rm End}_{0}(E)}+\theta)$ is computed by the second term $E_{2}^{p,q}$ of the spectral sequence of the double complex.
We have 
\[E_{1}^{p,q}=H^{p,q}_{B}(M, T_{M}^{1,0\ast})\oplus H^{p,q}_{B}(M, \C)\oplus H^{p,q}_{B}(M, T^{1,0}_{M})
\]
for $(p,q)=(0,0), (1,0), (0,1), (1,1)$.
By the basic Kodaira-vanishing theorem (see \cite{Noz}), we have $H^{0,0}_{B}(M,  T_{M}^{1,0}) =H^{1,1}_{B}(M,  T^{1,0\ast}_{M})=0$.
The differential $d_{1}:E_{1}^{p,q}\to E_{1}^{p+1,q}$ is induced by $\theta$.
Since  $\theta$ is defined by the identification of the basic holomorphic  bundle $ T^{1,0\ast}_{M}$ with the basic holomorphic bundle generated by basic holomorphic  $1$-forms, $d_{1}$ gives  isomorphisms
\[H^{0,0}_{B}(M, \C)\cong H^{1,0}_{B}(M, T^{1,0}), \qquad H^{0,0}_{B}(M, T^{1,0\ast})\cong H^{1,0}_{B}(M, \C)
\]
\[H^{0,1}_{B}(M,  T^{1,0\ast})\cong H^{1,1}_{B}(M, \C)\qquad {\rm and}\qquad H^{0,1}_{B}(M, \C)\cong H^{1,1}_{B}(M, T^{1,0})
\]
and $d_{1}=0$ on $H^{1,0}(M, T^{1,0\ast})$ and $H^{0,1}(M, T^{1,0})$.
Thus $E_{2}^{1,0}=H^{1,0}(M, T^{1,0\ast})$ and $E_{2}^{0,1}=H^{0,1}(M, T^{1,0})$.
By the Basic Serre duality (see \cite{Noz}),  $\dim H^{1,0}(M, T^{1,0\ast})=\dim H^{0,1}(M, T^{1,0})$.

The homeomorphism ${\mathcal M}^{s}_{Bflat }(E, D_{\xi})\cong {\mathcal M}^{s}_{Bharm}(E, D_{\xi})$ via the non-abelian Hodge correspondence is a diffeomorphism in this assumption (see \cite{Cor} and  also arguments in the proof of  \cite[Proposition 2.3]{Cor2}).
We regard $SU(1,1)\subset SL_{2}(\C)$ as a real form.
Then for ${\rm Hom}^s(\pi_{1}(M,x), SU(1,1))={\rm Hom}(\pi_{1}(M,x), SU(1,1))\cap{\rm Hom}^s(\pi_{1}(M,x), SL_{2}(\C))$.
Via the  natural map 
\[{\rm Hom}^s(\pi_{1}(M,x), SU(1,1))/SU(1,1)\to {\rm Hom}^{s}(\pi_{1}(M,x), SL_{2}(\C))/ SL_{2}(\C), 
\]
the  space ${\rm Hom}^s(\pi_{1}(M,x), SU(1,1))/SU(1,1) $ 
can be  regarded as a real  submanifold in  ${\rm Hom}^{s}(\pi_{1}(M,x), SL_{2}(\C))/ SL_{2}(\C)$ of dimension $2\dim_{\C} H^{0,0}_{B}(M, (T^{1,0\ast})^2)$ (see \cite{JM}).

\section{The (lifting of) Hitchin section}
Let  $M$ be a compact $3$-dimensional manifold equipped with a $\widetilde{SL_{2}(\R)}$-Sasakian structure $(\eta, w)$.
Then we can assume 
$M=\Gamma \backslash\widetilde{SL_{2}(\R)} $ with a left-invariant Sasakian structure $(T^{1,0}_{M}=\langle W\rangle,\eta)$.
We have 
\[\sqrt{-1}d\eta +w\wedge \bar{w}= 0\qquad
{\rm and}\qquad dw +2\sqrt{-1}\eta\wedge w=0
\]
We have $E_{-\frac{1}{\pi}}\cong T^{1,0}$.
Define a stable Higgs bundle $(E, \theta_{\alpha})$ such that $E=E_{-\frac{1}{2\pi}} \oplus E_{\frac{1}{2\pi}}$ and 
$\theta=\left(
\begin{array}{cc}
0 & 1 \\
\alpha  & 0
\end{array}
\right)w$ 
satisfying
\[\xi (\alpha)-2\sqrt{-1}(\alpha)=\overline{W}(\alpha)=0.
\]
We have $\alpha w\otimes w\in H^{1,0}_{B}(M, T^{1,0\ast }_M)$.
By the non-abelian Hodge correspondence (\cite{BK, BK2}), we have a Hermitian metric $H_{\alpha}$ on $E$ satisfying the  Hitchin equation.
By the similar argument on  usual cyclic Higgs bundles (see \cite{Ba, Li}), we can say that $H_{\alpha}= \left(
\begin{array}{cc}
h^{-1}_{\alpha} & 0\\
0  & h_{\alpha}
\end{array}
\right)$ for a positive real function $h$.
Thus we have a flat connection 
\begin{multline*}
d+\left(
\begin{array}{cc}
\sqrt{-1} (\eta-\frac{1}{2}d_{B}^{c}{\rm log} (h_{\alpha}))& 0 \\
0 & -\sqrt{-1}(\eta-\frac{1}{2}d_{B}^{c}{\rm log} (h_{\alpha}))
\end{array}
\right)\\
+\left(
\begin{array}{cc}
0 & h^{-2}_{\alpha} \\
h^2_{\alpha}\alpha & 0
\end{array}
\right)w+\left(
\begin{array}{cc}
0 & h^{2}_{\alpha} \overline{\alpha}\\
h^{-2}_{\alpha} & 0
\end{array}
\right)\overline{w}.
\end{multline*}
Let $\eta_{\alpha}=\eta-\frac{1}{2}d_{B}^{c}{\rm log} (h_{\alpha})$ and  $w_{\alpha}=h^{-2}_{\alpha}w+ h^2_{\alpha}\overline{\alpha}\overline{w}$.
We have 
\[\sqrt{-1}d\eta_{\alpha} +w_{\alpha}\wedge \overline{w_{\alpha}}= 0 \qquad
{\rm and}\qquad dw_{\alpha}+2\sqrt{-1}\eta_{\alpha}\wedge w_{\alpha}=0.
\]
By 
\[\sqrt{-1}d\eta +w\wedge \overline{w}= 0,
\]
the equation
\[\sqrt{-1}d\eta_{\alpha} +w_{\alpha}\wedge \overline{w_{\alpha}}= 0 \]
 is reduced to
\[-\sqrt{-1}\partial_{B}\bar\partial _{B}{\rm log} (h_{\alpha})-\sqrt{-1}(h_{\alpha}^{-4}-h_{\alpha}^4\vert \alpha\vert^{2}-1)d\eta=0.
\]
By the basic K\"ahler identities, we  have 
\[2\Delta_{B}{\rm log} (h_{\alpha})-(h_{\alpha}^{-4}-h_{\alpha}^4\vert \alpha\vert^{2}-1)=0.
\]
Since $T^{1,0\ast}_{M}$ is non-trivial as a basic holomorphic line bundle, $\alpha$ is non-constant and  has zeros.
For the usual Laplacian  $\Delta=-\xi^{2}- W\overline{W}$ on $A^{0}(M)$, we have $\Delta_{B}=\Delta$ on $A^{0}_{B}(M)$.
By  
\[\xi (\alpha)-2\sqrt{-1}(\alpha)=\overline{W}(\alpha)=0,
\]
we have $\Delta {\rm log}\vert \alpha\vert^2=0$ and hence we have
\[(\xi^{2}+ W\overline{W}){\rm log} (h^8_{\alpha}\vert \alpha^2\vert)=-\Delta{\rm log} (h^8_{\alpha}\vert \alpha^2\vert)=4h_{\alpha}^{-4}(h_{\alpha}^8\vert \alpha\vert^{2}-1)+4
\]
outside the Zeros of $\alpha$.
By the maximum principle, 
we have \[h_{\alpha}^8\vert \alpha\vert^{2}-1\le {\rm max}(h_{\alpha}^8\vert \alpha\vert^{2})-1<0.\]
This implies that $w_{\alpha}\wedge \overline{w_{\alpha}}=(h_{\alpha}^{-4}-h_{\alpha}^4\vert \alpha\vert^{2})w\wedge \overline{w}$ is non-zero everywhere (cf. \cite[Claim 6.1]{Li})  .

Thus, we obtain  a new $\widetilde{SL_{2}(\R)}$-Sasakian structure $(\eta_{\alpha}, w_{\alpha})$ on $M$.

By the above construction, we have the smooth family $(\eta_{\alpha}, w_{\alpha})$ of  $\widetilde{SL_{2}(\R)}$-Sasakian structures such that the space of  parameters $\alpha$ is  considered as $H^{1,0}_{B}(M, T^{1,0\ast}_M)$ and $(\eta_{0}, w_{0})=(\eta, w)$.
We note that by the basic Serre duality (see \cite{Noz}), $\alpha w\otimes w \in H^{1,0}_{B}(M, T^{1,0\ast}_M) $ is corresponding to a harmonic 
form $h^{4}_{\alpha}\alpha \bar{w}\otimes W\in A^{0,1}_{B}(M, T^{1,0}_M) $ and we have an isomorphism 
\[H^{1,0}_{B}(M, T^{1,0\ast}_M) \ni \alpha w\otimes w \mapsto [h^{4}_{\alpha}\alpha \bar{w}\otimes W]\in  H^{0,1}_{B}(M, T^{1,0}_M) .\]

We note that the Sasakian structure $(T^{1,0}_M,\eta) $ is quasi-regular.
By the basic Kodaira-vanishing theorem (see \cite{Noz}), we have $H^{0,0}_{B}(M,  T^{1,0}_M) =0$ and hence  $H^{0}(X,  T^{1,0}X) =0$.
Since $X$ is a complex $1$-dimensional orbifold, $H^{2}(X,  T^{1,0}X) =0$.
Thus, $(\eta_{\alpha}, w_{\alpha})$ gives  a differential deformation family $(\xi, T^{1,0}_{M_{\alpha}})$ of transverse holomorphic structures such that the Kodaira-Spencer map $T_{0}H^{1,0}_{B}(M, T^{1,0\ast}_M) \to H^{1}(M, \Theta)\cong H^{0,1}_{B}(M, T^{1,0}_M)$ is an isomorphism.

Denote $g(X)$ by the genus of  the topological surface $X$ and $n(X)$ by the number of orbifold points of $X$.
We can  compute 
\[\dim H^{1,0}_{B}(M, T^{1,0\ast}_M) =3g(X)-3+n(X)
\]
by the Riemann-Roch formula for $V$-bundles (see \cite{Kaw} and \cite{FS}).

We have obtained the smooth map  
\begin{multline*}
\sigma : H^{1,0}_{B}(M, T^{1,0\ast}_M)\ni \alpha w\otimes w\\
\mapsto \left(d+\left(
\begin{array}{cc}
-\sqrt{-1} \eta_{\alpha}& 0 \\
0 & \sqrt{-1}\eta_{\alpha}
\end{array}
\right), \left(
\begin{array}{cc}
0 & w_{\alpha} \\
\overline{w_{\alpha}} & 0
\end{array}
\right) \right)\in {\mathcal M}^{s}_{Bharm  }(E, D_{\xi}).
\end{multline*}
We have 
\[
\sigma\circ  \tau (\alpha w\otimes w)={\rm det} \left(
\begin{array}{cc}
0 & h^{-2}_{\alpha}w \\
h^2_{\alpha}\alpha w & 0
\end{array}
\right)=\alpha w\otimes w.
\] 
Thus the $\sigma :  H^{1,0}_{B}(M, T^{1,0\ast}_M)\to {\mathcal M}^{s}_{Bharm  }(E, D_{\xi})$ is a section of  $\tau :{\mathcal M}^{s}_{Bharm  }(E, D_{\xi})\to H^{0,0}_{B}(M, (T^{1,0\ast}_M)^2)=H^{1,0}_{B}(M, T^{1,0\ast}_M)$.
This implies that $\sigma :  H^{1,0}_{B}(M, T^{1,0\ast}_M)\to {\mathcal M}^{s}_{Bharm  }(E, D_{\xi})$ is a closed  smooth embedding.

Via  diffeomorphisms 
\[{\rm Hom}^{s}(\pi_{1}(M,x), SL_{2}(\C))/ SL_{2}(\C)\cong {\mathcal M}^{s}_{flat }(SL_r)  \supset{\mathcal M}^{s}_{Bflat }(E, D_{\xi})\cong {\mathcal M}^{s}_{Bharm}(E, D_{\xi}),\]
$\sigma :  H^{1,0}_{B}(M, T^{1,0\ast}_M)\to {\mathcal M}^{s}_{Bharm  }(E, D_{\xi})$ is regarded as a map $H^{1,0}_{B}(M, T^{1,0\ast}_M)\to {\rm Hom}^s(\pi_{1}(M,x), SU(1,1))/SU(1,1)$.
By \[\dim {\rm Hom}^s(\pi_{1}(M,x), SU(1,1))/SU(1,1)=2\dim_{\C} H^{1,0}_{B}(M, T^{1,0\ast}_M),\]  $H^{1,0}_{B}(M, T^{1,0\ast}_M)$ is embbeded into \[ {\rm Hom}^s(\pi_{1}(M,x), SU(1,1))/SU(1,1)\] as a connected component.

We consider the identification ${\mathcal S}_{(-,0)}(M)/\sim={\mathcal R}(\pi_{1}(M,x), \widetilde{SL_{2}(\R)})/\widetilde{SL_{2}(\R)}$ as in Corollary \ref{SLsas}.
The map
\begin{multline*}
{\mathcal S}_{(-,0)}(M)/\sim\ni [\eta, \sqrt{-1}w\otimes W-\sqrt{-1}\overline{w}\otimes \overline{W}]\\
\mapsto \left[d+\left(
\begin{array}{cc}
-\sqrt{-1} \eta& w \\
\overline{w} & \sqrt{-1}\eta
\end{array}
\right)\right]\in {\mathcal M}^{s}_{flat }(SL_r)
\end{multline*}
is identified with the natural map
\begin{multline*}\pi: {\mathcal R}(\pi_{1}(M,x), \widetilde{SL_{2}(\R)})/\widetilde{SL_{2}(\R)}\\
\to {\rm Hom}^s(\pi_{1}(M,x), SU(1,1))/SU(1,1)\hookrightarrow  {\rm Hom}^{s}(\pi_{1}(M,x), SL_{2}(\C))/ SL_{2}(\C)
\end{multline*}
induced by the covering map $p: \widetilde{SL_{2}(\R)})\to SL_{2}(\R)\cong SU(1,1)$.
Via this map, the map
\[H^{1,0}_{B}(M, T^{1,0\ast}_M) \ni \alpha w\otimes w \mapsto [\eta_{\alpha}, \sqrt{-1}w_{\alpha}\otimes W_{\alpha}-\sqrt{-1}\overline{w_{\alpha}}\otimes \overline{W_{\alpha}}]\in {\mathcal S}_{(-,0)}(M)/\sim
\]
can be seen as a lifting of the embedding \[H^{1,0}_{B}(M, T^{1,0\ast}_M)\hookrightarrow {\rm Hom}^s(\pi_{1}M, SU(1,1))/SU(1,1).\]
Summarizing the above arguments, we obtain the following results.
\begin{thm}\label{secdef}
Let $(M,\, T^{1,0}_{M} ,\, \eta) $ be a compact  $3$-dimensional Sasakian manifold such that $(\eta, \Phi_{\xi})\in {\mathcal S}_{(-,0)}(M)$.
Then there exists a smooth family $(\eta_{\alpha}, w_{\alpha})$ of $\widetilde{SL_{2}(\R)}$-Sasakian structures parametrized by $\alpha w\otimes  w\in H^{1,0}_{B}(M, T^{1,0\ast}_M)$ such that:
\begin{itemize}
\item $[\eta_{0}, \sqrt{-1}w_{0}\otimes W_{0}-\sqrt{-1}\overline{w_{0}}\otimes \overline{W_{0}}]=[\eta, \Phi_{\xi}]\in {\mathcal S}_{(-,0)}(M)/\sim$.
\item  $(\xi,  \langle W_{\alpha}\rangle )$ is a differential  deformation family of the transverse complex structure $(\xi, T^{1,0}_{M})$ such that the Kodaira-spencer map at $0$ is an isomorphism.
\item The  map
\begin{multline*}H^{1,0}_{B}(M, T^{1,0\ast}_M)\ni \alpha w\otimes  w
\mapsto \left[d+\left(
\begin{array}{cc}
\sqrt{-1} \eta_{\alpha}& w_{\alpha} \\
\overline{w_{\alpha}} & -\sqrt{-1}\eta_{\alpha}
\end{array}
\right)\right]\in {\mathcal M}^{s}_{flat }(SL_r)
\end{multline*}
induces a diffeomorphism from the  vector space  $H^{1,0}_{B}(M, T^{1,0\ast}_M)$ onto a connected component of ${\rm Hom}^s(\pi_{1}M, SU(1,1))/SU(1,1)$.
\item Via the identification  ${\mathcal S}_{(-,0)}(M)/\sim={\mathcal R}(\pi_{1}M, \widetilde{SL_{2}(\R)})/\widetilde{SL_{2}(\R)}$ as in Corollary \ref{SLsas}, 
 the map
\begin{multline*}
H^{1,0}_{B}(M, T^{1,0\ast}_M) \ni \alpha w\otimes w \\
\mapsto [\eta_{\alpha}, \sqrt{-1}w_{\alpha}\otimes W_{\alpha}-\sqrt{-1}\overline{w_{\alpha}}\otimes \overline{W_{\alpha}}]\in {\mathcal S}_{(-,0)}(M)/\sim 
\end{multline*}
can be regarded as  a lifting  
\[\tilde\sigma: H^{1,0}_{B}(M, T^{1,0\ast}_M)\to {\mathcal R}(\pi_{1}(M,x), \widetilde{SL_{2}(\R)})/\widetilde{SL_{2}(\R)} \]
 of the injection $H^{1,0}_{B}(M, T^{1,0\ast}_M)\hookrightarrow {\rm Hom}^s(\pi_{1}(M,x), SU(1,1))/SU(1,1)$ induced  by $\sigma $ for the map 
 \[\pi: {\mathcal R}(\pi_{1}(M,x), \widetilde{SL_{2}(\R)})/\widetilde{SL_{2}(\R)}\to {\rm Hom}^s(\pi_{1}(M,x), SU(1,1))/SU(1,1)
\]
induced by the covering map $p: \widetilde{SL_{2}(\R)}\to SL_{2}(\R)\cong SU(1,1)$.
\end{itemize}

\end{thm}

\begin{remark}
For the unitary matrix $P=\frac{1}{\sqrt{2}}\left(
\begin{array}{cc}
 1& 1 \\
-\sqrt{-1} & \sqrt{-1}
\end{array}
\right)$, we have $SU(1,1)=P^{-1}SL_{2}(\R)P$.
Thus, for the natural global frame $e_{-},e_{+}$ of $E=E_{-\frac{1}{2\pi}} \oplus E_{\frac{1}{2\pi}}$, the global  frame 
$\frac{1}{\sqrt{2}}(h_{\alpha}^{\frac{1}{2}}e_{-}+h_{\beta}^{-\frac{1}{2}}e_{+}), \frac{\sqrt{-1}}{\sqrt{2}}(h^{\frac{1}{2}}e_{-}-h^{-\frac{1}{2}}e_{+})$ gives a real structure on the flat connection  corresponding to the basic Higgs bundle $(E,\theta_{\alpha})$.
That is we have the real flat bundle 
\[\left(E^{\R}=M\times \R^{2}, D_{\alpha}^{\R}=d+P\left(
\begin{array}{cc}
\sqrt{-1} \eta_{\alpha}& w_{\alpha} \\
\overline{w_{\alpha}} & -\sqrt{-1}\eta_{\alpha}
\end{array}
\right)P^{-1}\right).\]

Define the symmetric bilinear form $C$ on $E$ by the matrix presentation  $\left(
\begin{array}{cc}
 0& 1 \\
1 & 0
\end{array}
\right)$ for the global  frame  $e_{-},e_{+}$.
Then, $C$ is non-degenerate, basic holomorphic and satisfying $C(\theta_{\alpha}\otimes id_{E})=C(id_{E}\otimes \theta_{\alpha})$.
We note that the real flat bundle $(E^{\R}, D_{\alpha}^{\R})$ corresponds to $(E,\theta_{\alpha}, C)$ (see \cite{LM}).

\end{remark}

\section{Geometry of ${\mathcal R}(\pi_{1}M, \widetilde{SL_{2}(\R)})$}
Let $(M,\, T^{1,0}_{M} ,\, \eta) $ be a compact  $3$-dimensional Sasakian manifold such that $(\eta, \Phi_{\xi})\in {\mathcal S}_{(-,0)}(M)$.

Consider the algebraic variety  ${\rm Hom}(\pi_{1}(M,x), SL_{2}(\C))$.
By the Goldman-Millson theory \cite{GM} and the almost formality in \cite{KaA}, we can say that the analytic space  ${\rm Hom}(\pi_{1}(M,x), SL_{2}(\C))$ is smooth.
Thus, ${\rm Hom}(\pi_{1}(M,x), SU(1,1))$ is a smooth manifold.
Consider  the map \[p_{\ast}:{\rm Hom}( \pi_{1}(M,x), \widetilde{SL_{2}(\R)})\to  {\rm Hom}(\pi_{1}(M,x), SU(1,1))\] induced by 
the covering   $p:\widetilde{SL_{2}(\R)} \to SU(1,1)$.
Then by the lifting property of $p_{\ast}$  (see \cite[Lemma 2.2]{Gol}), we can say that ${\rm Hom}( \pi_{1}(M,x), \widetilde{SL_{2}(\R)})$ is also a smooth manifold.
It is known that $ {\mathcal R}(\pi_{1}(M,x), \widetilde{SL_{2}(\R)})$ is an open set in ${\rm Hom}( \pi_{1}(M,x), \widetilde{SL_{2}(\R)})$ (see \cite{Wei}). 

We notice that the action of $SU(1,1)$ (resp. $\widetilde{SL_{2}(\R)}$)  on ${\rm Hom}(\pi_{1}M, SU(1,1))$  (resp. ${\rm Hom}( \pi_{1}M, \widetilde{SL_{2}(\R)})$) is reduced to an effective action of $PSU(1,1)$.
The action on ${\rm Hom}^s(\pi_{1}M, SU(1,1))$  is proper (see \cite{JM}) and free  as a $PSU(1,1)$-action.
We can say that the action on  ${\rm Hom}( \pi_{1}M, \widetilde{SL_{2}(\R)})$ is also proper and free  as a $PSU(1,1)$-action.
Thus, the quotient space ${\mathcal R}(\pi_{1}(M,x), \widetilde{SL_{2}(\R)})/\widetilde{SL_{2}(\R)}$ is a smooth manifold.
Now the lifting 
 \[\tilde\sigma: H^{1,0}_{B}(M, T^{1,0\ast}_M)\to {\mathcal R}(\pi_{1}(M,x), \widetilde{SL_{2}(\R)})/\widetilde{SL_{2}(\R)} \]
can be seen as a smooth section of a smooth map $\pi :{\mathcal R}(\pi_{1}(M,x), \widetilde{SL_{2}(\R)})/\widetilde{SL_{2}(\R)}\to {\rm Hom}^s(\pi_{1}(M,x), SU(1,1))/SU(1,1)$ from the connected component $H^{1,0}_{B}(M, T^{1,0\ast}_M)$ of ${\rm Hom}^s(\pi_{1}(M,x), SU(1,1))/SU(1,1)$.

\begin{thm}
The map 
$\tilde\sigma: H^{1,0}_{B}(M, T^{1,0\ast}_M)\to {\mathcal R}(\pi_{1}(M,x), \widetilde{SL_{2}(\R)})/\widetilde{SL_{2}(\R)}$ 
induces  a diffeomorphism from the vector space $H^{1,0}_{B}(M, T^{1,0\ast}_M)$  onto a connected component of ${\mathcal R}(\pi_{1}(M,x), \widetilde{SL_{2}(\R)})/\widetilde{SL_{2}(\R)}$.

\end{thm}

\section{Standard cases}
Let  $\Sigma$ be a compact Riemannian surface    of genus $g(\Sigma)\ge 2$.
We uniformize $\Sigma =\Gamma  \backslash \bf H$ such that $\bf H$ is the  upper half plane    and $\Gamma$ is a cocompact discrete subgroup in $PSL_{2}(\R)$.
Take the homogeneous space  $M=\Gamma \backslash PSL_{2}(\R)$.
$M$ admits a natural $\widetilde{SL_{2}(\R)}$-Sasakian structure induced by a left-invariant Sasakian structure.
We notice that $M$ is a principal $S^{1}$-bundle over $\Sigma$ associated with the right $PSO(2)$-action on $\Gamma \backslash PSL_{2}(\R)$,  $M=\pi_1(M, x)\backslash \widetilde{SL_{2}(\R)}$, $\pi_1(M, x)$ is the pullback of  $\Gamma\subset PSL_{2}(\R)$ for the covering $\widetilde{SL_{2}(\R)}\to  PSL_{2}(\R)$
 and we have the central extension 
\[
\xymatrix{
1\ar[r]&\Z\ar[r]&\pi_1(M, e)\ar[r]&
\Gamma\ar[r]&1.}
\]
The representation $\pi_1(M, e)\to  SU(1,1)$ induced by the covering \[\widetilde{SL_{2}(\R)}\to  SL_{2}(\R)\cong SU(1,1)\] sends the central subgroup $\Z$ onto $\{\pm I\}\subset SL(1,1)$ where $I$ is the identity matrix.
We say that a representation $\rho\in {\rm Hom}(\pi_1(M, x), SL_{2}(\C))$ is odd if $\rho$ sends the central subgroup $\Z$ onto $\{\pm I\}\subset SL_{2}(\C)$.
For the space  ${\rm Hom}^{s,odd}(\pi_1(M, x), SL_{2}(\C))$ of simple odd representations, 
 we have 
 \[{\rm Hom}^{s,odd}(\pi_1(M, x), SL_{2}(\C))/SL_{2}(\C)={\mathcal M}^{s}_{Bharm  }(E, D_{\xi})\] for $(E=M\times \C^{2}, D=d+\omega_{\widetilde{SL_{2}(\R)}})$ as in Section \ref{modul} (see \cite[Example 4.5]{Ka}). 
In this case, the section $\sigma :  H^{1,0}_{B}(M, T^{1,0\ast}_M)\to {\mathcal M}^{s}_{Bharm  }(E, D_{\xi})$ gives a connected component of \[{\rm Hom}^{s,odd}(\pi_1(M, x), SU(1,1))/SU(1,1)\] for the space ${\rm Hom}^{s,odd}(\pi_1(M, x), SU(1,1))$ of odd representations into $SU(1,1)$.
Hence, we obtain a connected component of ${\rm Hom}^{s}(\Gamma, PSL_{2}(\R))/PSL_{2}(\R)$  by the pullbacks.
Here $\Gamma $ is the fundamental group of the compact Riemannian surface $\Sigma$.
This means that we obtain the original Hitchin section in \cite{Hit}.

We can obtain "other" sections by the following way.
We have a lifting $\Gamma\subset SL_{2}(\R)$ of $\Gamma \backslash PSL_{2}(\R)$ (see \cite{Pat} or solving the Hitchin equation of the uniformizing Higgs bundle $(K_{\Sigma}^{-\frac{1}{2}}\oplus K_{\Sigma}^{\frac{1}{2}},Id_{K_{\Sigma}^{-1}})$).
Consider the Sasakian manifold  $M_{2}=\Gamma \backslash SL_{2}(\R)$ which is different from $M$.
Then  $\pi_1(M_{2}, x)$ is the subgroup in   $\pi_1(M, x)$ 
defined by  the central extension 
\[
\xymatrix{
1\ar[r]&2\Z\ar[r]&\pi_1(M_{2}, e)\ar[r]&
\Gamma\ar[r]&1.}
\]
We obtain a connected component ${\rm Hom}^s(\pi_{1}(M_{2},x), SU(1,1))/SU(1,1)$ by the section $\sigma :  H^{1,0}_{B}(M_{2}, T^{1,0\ast}_{M_{2}})\to {\mathcal M}^{s}_{Bharm  }(E, D_{\xi})$.

\end{document}